\documentclass[leqno,12pt]{article}

\usepackage{epsfig}

\usepackage{amsmath}
\usepackage{amscd}
\usepackage{amsopn}
\usepackage{amsthm}
\usepackage{amsfonts,amssymb}
\usepackage{bbm}
\usepackage{latexsym}

\usepackage{color}

\makeatletter
\@addtoreset{equation}{section}
\makeatother

\newtheorem{lemma}{LEMMA}[section]
\newtheorem{proposition}[lemma]{PROPOSITION}
\newtheorem{corollary}[lemma]{COROLLARY}
\newtheorem{theorem}[lemma]{THEOREM}

\newtheorem{remarks}[lemma]{REMARKS}

\newtheorem{definitions}[lemma]{DEFINITIONS}

\newcommand{\real} {\mathbbm{R}}

\newcommand{\nat}{\mathbbm{N}}

\renewcommand{\a}{\alpha}
\renewcommand{\b}{\beta}

\newcommand{\g}{\gamma}

\newcommand{\vp}{\varphi}
\newcommand{\ve}{\varepsilon}

\newcommand{\realn}{{\real^n}}

\newcommand{\on}{\quad\text{ on }}
\newcommand{\und}{\quad\mbox{ and }\quad}
\newcommand{\inv}{^{-1}}
\newcommand{\ov}{\overline}

\newcommand{\W}{\mathcal W}  
  
\newcommand{\C}{\mathcal C}  

\newcommand{\F}{\mathcal F}

\renewcommand{\H}{{\mathcal H}}
\newcommand{\B}{\mathcal B}

\newcommand{\M}{\mathcal M}
\newcommand{\A}{\mathcal A}

\newcommand{\itemframe}%
{\setlength{\parskip}{10pt}\begin{enumerate} \setlength{\topsep}{10pt}%
\setlength{\itemsep}{15pt}\setlength{\parsep}{5pt}}

\newcommand{\diam}{\operatorname*{diam}}

\newcommand{\intoi}{\int_0^\infty}

\newcommand{\sumj}{\sum_{j=1}^\infty}
\newcommand{\cupj}{\bigcup_{j=1}^\infty}

\newcommand{\zb}{Z_\b}

\title{Semipolar sets and intrinsic Hausdorff measure}
\date{}
 
\author{Wolfhard Hansen and Ivan Netuka}

\begin{document}
\maketitle 

\begin{abstract}
Given a ``Green function'' $G$ on a locally compact space $X$ with countable
base, a Borel set $A$ in $X$ is called \emph{$G$-semipolar}, if there is no measure $\nu\ne 0$ supported
by $A$  such that $G\nu:=\int G(\cdot,y)\,d\nu(y)$ is a continuous  real function on $X$.
Introducing an \emph{intrinsic Hausdorff measure $m_G$}  using $G$-balls 
$B(x,\rho):=\{y\in X\colon G(x,y)>1/\rho\}$, it is shown that 
every set $A$ in $X$ with $m_G(A)<\infty$ is contained in a~$G$-semipolar
Borel set. 

This is of interest, since $G$-semipolar sets are semipolar in the potential-theoretic
sense (countable unions of totally thin sets, hit by a corresponding process at most countably many 
times) provided $G$ is really a Green function  for a~harmonic space or, more generally,   a~balayage space.

For classical potential theory  
and Riesz potentials on $\mathbbm R^n$ or, more generally, for   Green functions 
on a metric measure space $(X,d,\mu)$ (where balls are relatively compact) given by a~continuous heat kernel 
$(x,y,t)\mapsto p_t(x,y)$ with upper and lower bounds
 of the form~$t^{-\alpha/\beta}\Phi_j(d(x,y)t^{-1/\beta})$, $j=1,2$, the intrinsic Hausdorff measure is equivalent 
to an ordinary Hausdorff measure~$m_{\alpha-\beta}$.

It is shown that for the corresponding space-time situation on $X\times \mathbbm R$ (heat equation 
on $\mathbbm R^n \times \real$ in the classical case of the Gauss-Weierstrass kernel) the intrinsic Hausdorff
measure is equivalent to an anisotropic Hausdorff measure~$m_{\alpha,\beta}$ (with $\alpha=n$ and $\beta=2$
 for the heat equation).

In particular, our result solves an open problem for the heat equation (which was the initial motivation
for the paper).

Keywords: Heat equation; metric measure space; 
heat kernel; balayage space; 
Green function; 
Hausdorff measure; 
semipolar set;
space-time process

MSC 2010: 28A78; 31C15; 31E05; 31C12; 31D05; 35J08; 35K08; 60J45; 60J60; 60J75

\end{abstract}

\section{Introduction}\label{intro}
In 1985, S.J.\,Taylor and N.A.\,Watson published a paper on 
a Hausdorff measure classification of polar sets for the heat equation
\begin{equation*} 
\sum_{i=1}^n \frac{\partial^2 u}{\partial x_i^2} - \frac {\partial u}{\partial t}=0
\end{equation*} 
on $\real^{n+1}$; see \cite{TW}. To that end they introduced an 
anisotropic (parabolic) measure~$m_P$ of Hausdorff type  (their notation is 
$\mathcal P-\Lambda^n-m$) defined as follows. For $\rho>0$,
$x' \in \real^{n+1}$, $A\subset \real^{n+1}$ and $\delta>0$,  let
\begin{equation}\label{def-P}
 P(0,\rho) :=\bigl [-\rho/2,\rho/2\bigl]^n\times \bigl[-\rho^2/2,\rho^2/2\bigr], \qquad P(x',\rho):=x'+P(0,\rho),
\end{equation} 
\begin{equation*} 
 m_P^\delta(A):=\inf\biggl\{\sumj \bigl(\diam P(x_j',\rho_j)\bigr)^n\colon
                    A\subset \cupj P(x_j',\rho_j), \, \diam P(x_j',\rho_j)<\delta\biggr\},
\end{equation*} 
\begin{equation*} 
  m_P(A):=\lim_{\delta\to 0}   m_P^\delta(A).
\end{equation*} 
They showed that $A$ is polar if $m_P(A)=0$ (\cite[Theorem 1]{TW}), 
and noted that 
  Borel sets $A\subset \realn\times \{0\}$ of strictly positive finite 
$n$-dimensional Lebesgue measure are nonpolar, but satisfy $m_P(A)<\infty$. 
Since such sets are semipolar, they raised the problem, if \emph{every} set $A$ in
$\real^{n+1}$ satisfying $m_P(A)<\infty$ is semipolar.

In this paper, we shall give an affirmative answer to this question. In fact,
we shall prove such a result in the abstract setting of locally compact space $X$ with countable base,
 where we consider function kernels~$G$ having simple regularity properties
which are satisfied by Green functions for a wide class of elliptic or parabolic second order partial differential 
operators (leading to harmonic spaces; see, for example, \cite[Section 7]{GH1}) as well as for 
rather general jump processes (leading to balayage spaces; see \cite{BH}).  
Special attention is given to an application on metric measure spaces; see Section \ref{section-metric}. 

The reader who is not familiar with or interested in this generality may suppose that
we mainly consider the case of the heat equation, classical potential theory, and Riesz potentials.

The clue will be the introduction of an intrinsic
measure $m_G$ of Hausdorff type using  \emph{$G$-balls}
\begin{equation*}
                              B(x,\rho):=\{y\in X\colon G(x,y)>\rho\inv\}, \qquad x\in X,\,\rho>0,
\end{equation*} 
and defining, for sets $A$ in $X$ and  $\delta>0$,
\begin{equation*} 
         m_G^\delta (A):=\inf
         \biggl\{ \sumj \rho_j\colon A\subset \cupj B(x_j,\rho_j),\, 0<\rho_j<\delta\bigg\}, \qquad
         m_G(A):=\lim_{\delta\to 0} m_G^\delta(A).
\end{equation*} 
A Borel set $A$ in $X$ will be called \emph{$G$-semipolar} if there is no measure $\mu\ne 0$
on $X$ such that $\mu(X\setminus A)=0$ and $G\mu:=\int G(\cdot,y)\,d\mu(y)$ is a~continuous 
real function.

Under simple assumptions on the regularity of $G$ we shall prove the following;
  see  properties (i) -- (iii) in Section \ref{G-section} 
and Theorem \ref{main-result}.

\begin{theorem}\label{intro-theorem}
If a set $A$ in $X$ satisfies  $m_G(A)<\infty$, then it is contained in a~$G$-semipolar Borel set.
\end{theorem} 

In fairly general potential-theoretic settings, semipolar sets (countable unions of totally sets,
sets which a corresponding process hits (almost surely) 
 at most countably many times) can be characterized
by not supporting continuous real potentials. Therefore, assuming that $G$ is really a Green
function, we have the following consequence; see properties (i$'$)-- (iv) in Section \ref{habasp} 
and Corollary \ref{m-semipolar}.

\begin{corollary}\label{intro-corollary}
If $A$ is a set in $X$ such that $m_G(A)<\infty$, then $A$ is semipolar.
\end{corollary}

Whereas the proof of Theorem \ref{intro-theorem} is rather involved and the characterization 
of semipolar sets is very subtle, there are many situations,
where it is easy to see that  every set $A$ in $X$ satisfying $m_G(A)=0$ is  contained in a ($G$-)polar Borel set;
see Corollaries \ref{m-polar} and \ref{m0polar}.

In classical potential theory on $\real^n$, every semipolar set is polar.
In the case \hbox{$n\ge 3$},
Theorem \ref{intro-theorem} reduces to the well known fact
that sets  $A$ in $\realn$ having finite $(n-2)$-dimensional Hausdorff measure
are polar: It suffices to take 
\begin{equation*} 
G(x,y):=|x-y|^{2-n}, \qquad x,y\in \realn,
\end{equation*} 
 and to observe that
$G$-balls $B(x,\rho)$ are  Euclidean balls with center $x$ and radius~$\rho^{n-2}$.
In the case $n=2$, Theorem \ref{intro-theorem} implies that sets $A$ in $\real^2$
having finite $\phi$-Hausdorff measure, where $\phi(t):=\log^+ t\inv$, are polar;
cf., for example,  \cite[Theorem~5.9.4]{gardiner-armitage}.

Taking, for the heat equation, 
\begin{equation}\label{def-W}
G(x',y'):=G_0(x'-y'), \qquad x',y'\in\real^{n+1} ,
\end{equation} 
where
\begin{equation}\label{G-heat}
                 G_0(x'):=1_{(0,\infty)}(t)\cdot t^{-n/2} \exp
\biggl(-\frac{ |x|^2}{4t}\biggr),
\qquad   x'=(x,t) \in\real^n\times \real,
\end{equation} 
we shall prove (Corollary \ref{PGP}) that, 
\begin{equation}\label{P-G-equiv}
  (2n)^{-n}  m_P \le m_G \le (2/n)^{n/2} m_P . 
\end{equation} 
So Corollaries \ref{intro-corollary} and \ref{m0polar} 
show that sets $A$ in $\real^{n+1}$ are semipolar, polar respectively
if $m_P(A)<\infty$, $m_P(A)=0$ respectively. 
 
In the  space-time setting given by general heat kernels 
$(x,y,t)\mapsto p_t(x,y)$ 
on metric measure spaces $(X,d,\mu)$ 
the intrinsic Hausdorff measure  will be equivalent 
to an anisotropic Hausdorff measure $m_{\a,\b}$ on $X':=X\times \real$;
see Section \ref{section-metric}.

In detail, our paper is organized as follows. We start with a short Section~\ref{sec-comp}
on   comparison of measures of Hausdorff type on arbitrary sets. In Section~\ref{G-section},
we prove Theorem~\ref{main-result} and show that sets $A$ with $m_G(A)=0$ are contained
in $G$-polar Borel sets.
 In Section~\ref{habasp}, we apply our general results to harmonic spaces and  balayage spaces
(diffusions and jump processes). In Section \ref{standard}, we briefly discuss the  results
in the standard settings of  classical potential theory, heat equation and Riesz potentials. 
In Section~\ref{space-time},
we prove results for  rather general space-time settings. In Section~\ref{section-metric},  
we finally apply our results to heat kernels on metric measure spaces covering diffusions and
jump processes on manifolds and fractals.

\section{Comparison of measures of Hausdorff type}\label{sec-comp}

To prove (\ref{P-G-equiv}) and similar estimates it will be helpful to
have a general comparison result for measures of Hausdorff type.
To that end let us consider an arbitrary set~$X$ and let $\F$ denote the set of
all mappings~$F$ which associate to all $x\in X$ and $\rho>0$ some
subset $F(x,\rho)$ of $X$. For all $F\in \F$, $\eta>0$ and $A\subset X$, 
we define
\begin{equation*} 
m_{\eta,F}^\delta (A):=\inf
         \biggl\{ \sumj \rho_j^\eta \colon A\subset \cupj F(x_j,\rho_j),\,
        x_j\in X, \, 0< \rho_j<\delta\bigg\}, \qquad \delta>0,
\end{equation*}
and 
\begin{equation*} 
         m_{\eta,F}(A):=\lim_{\delta\to 0} m_{\eta,F}^\delta(A).
\end{equation*} 

\begin{proposition}\label{general-comparison}
Let $\eta,\tilde \eta\in (0,\infty)$, $F,\tilde F\in \F$ and $\kappa>0$
such that, for all $x\in X$ and $\rho>0$, there exists $z\in X$
with $       F(x,\rho)\subset \tilde F(z,\kappa  \rho^{\eta/\tilde
  \eta})$. Then 
\begin{equation*} 
m_{\tilde \eta,\tilde F}\le \kappa^{\tilde\eta}m_{\eta,F}.
\end{equation*} 
\end{proposition} 

\begin{proof} Let $A\subset X$, $x_j\in X$ and $0<\rho_j<\delta$, $j\in\nat$,
 such that $A\subset \bigcup_{j=1}^\infty F( x_j,\rho_j)$. Then there are
 $z_j\in X$ such that 
  $A\subset \bigcup_{j=1}^\infty \tilde F(z_j,\kappa
 \rho_j^{\eta/\tilde \eta})$.
 Taking $\tilde \delta:= \kappa\delta^{\tilde \eta/\eta}$ we conclude that
\begin{equation*}
 m_{\tilde \eta,\tilde F}^{\tilde\delta}(A)
\le   \sum_{j=1}^\infty  (\kappa \rho_j^{\eta/\tilde \eta})^{\tilde \eta}
= \kappa^{\tilde \eta}  \sum_{j=1}^\infty   \rho_j^{\eta} .
\end{equation*} 
This implies that
\begin{equation*} 
  m_{\tilde \eta,\tilde F}^{\tilde \delta} (A) \le  \kappa^{\tilde \eta} m_{\eta,F}^\delta(A).
\end{equation*} 
 The proof is completed letting $\delta$ tend to $0$.
\end{proof}

 \section{Hausdorff measure with respect to a kernel}\label{G-section}

Let $X$ be a locally compact space with countable base and let 
$G$ be a  positive numerical function on $X\times X$  having the following properties:
\begin{itemize}
\item[\rm (i)] For every $y\in X$,   $G(\cdot,y)$ is lower semicontinuous 
and 
$\limsup_{x\to y} G(x,y)=\infty$.
\item[\rm (ii)] $G$ is continuous outside the diagonal $\Delta$ and Borel measurable on $\Delta$.
\item[\rm (iii)] For every compact $K$ in $X$, there exists a compact   $L$ in $X$ 
such that $G$ is bounded on $(X\setminus L)\times K$.
\end{itemize} 

Let $\M(X)$ denote the set of all  (positive Radon) measures  on $X$. For $\mu\in \M(X)$, we define
\begin{equation*}
G\mu(x):=\int G(x,y)\,d\mu(y), \qquad x\in X.
\end{equation*} 
By (i) and Fatou's lemma, these functions $G\mu$ are lower semicontinuous. 
This implies that, for   $\mu,\nu\in \M(X)$ such that $\nu\le \mu$
and $G\mu$ is continuous and real, the function $G\nu$ is continuous 
and real.

The following definitions are justified by characterizations of polar and semipolar
sets related to the Laplace equation, the heat equation, harmonic spaces and balayage spaces
(associated with diffusions and jump processes), where positive constants are superharmonic  
and there exists a~Green function. 

\begin{definitions}
Let us say that a Borel set $A$ in $X$ is   \emph{$G$-polar}
if there is \emph{no} measure $\mu\ne 0$ such that $\mu(X\setminus A)=0$
and $G\mu$ is bounded.

Let us say that a  Borel set $A$ in $X$ is  \emph{$G$-semipolar}
if there is \emph{no} measure~$\mu\ne 0$ such that $\mu(X\setminus A)=0$
and $G\mu$ is continuous and real. 
\end{definitions}

\begin{remarks} \label{G-semipolar}{\rm
1.   A Borel set $A$ is  $G$-semipolar if there is 
no measure~$\mu\ne 0$ such that $\mu(X\setminus A)=0$
and $G\mu$ is continuous and \emph{bounded}, which implies
that every $G$-polar set is $G$-semipolar.

 Indeed, suppose that $A$ is not $G$-semipolar.
Then there exists a measure $\mu\ne 0$ such that $\mu(X\setminus A)=0$
and $G\mu$ is continuous and real.  Of course, we may choose a~compact $K$
in $A$ such that $\mu(K)>0$. Let $\nu=1_K\mu$. Then $G\nu$ is continuous and real.
In particular, $G\nu$ is bounded on every compact in $X$.
Choosing a compact neighborhood $L$ of $K$ such that $G$ is bounded by some
constant $a$ on $(X\setminus L)\times K$, we obtain that
$                           G\nu(x)=\int G(x,y)\,d\nu(y)\le a\nu(K)<\infty$ 
 for every $x\in X\setminus L$. 

2. If $A_j$, $j\in\nat$, are  Borel sets in $X$, then the union $A:=\cupj A_j$ is $G$-semipolar if
and only if every $A_j$ is $G$-semipolar. 

Indeed, suppose that the sets $A_j$, $j\in\nat$, are $G$-semipolar and consider $\mu\in\M(X)$
such that $\mu(X\setminus A)=0$ and $G\mu$ is continuous and real. Defining  $\mu_j:=1_{A_j}\mu$ we know
that  $\mu_j(X\setminus A_j)=0$ and $G\mu_j$ is continuous and real, hence $\mu_j=0$ for  $j\in\nat$.
So $\mu=0$, and we see that $A$ is $G$-semipolar. The converse is trivial.
}\end{remarks}

We define \emph{$G$-balls} $B(x,\rho)$ by 
\begin{equation*} 
         B(x,\rho):=\{y\in X\colon  G(x,y)> \rho\inv\}, \qquad x\in X, \, \rho>0.
\end{equation*} 
Let us observe that obviously, for all such balls $B(x,\rho)$ (which, by (ii), are Borel sets) and
  measures $\nu$ on $X$,
\begin{equation}\label{nu-B-G}
\nu(B(x,\rho))\le \rho\int_{B(x,\rho)} G(x,y)\,d\nu(y)\le \rho G\nu(x).
\end{equation} 
Moreover, we define an \emph{intrinsic Hausdorff measure} $m_G$ as
follows. 
For every subset~$A$ of~$X$, let
\begin{equation*} 
   m_G^\delta(A):=\inf\big\{\sumj  \rho_j\colon 
                      A\subset \bigcup_{j=1}^\infty B(x_j,\rho_j),\ x_j\in X, \, 0<\rho_j<\delta\big\},
                      \qquad \delta>0,
\end{equation*} 
\begin{equation*} 
m_G(A):=\lim_{\delta\to 0} m_G^\delta(A).
\end{equation*} 
Let us note right away that, for every subset $A$ of $X$, 
\begin{equation}\label{aa-prime}
 m_G(A)=m_G(\tilde A) \qquad\mbox { for some Borel set $\tilde A$ containing $A$}.
\end{equation} 
If $k\in\nat$ and $\a_k:=m_G^{1/k}(A)<\infty$, 
 we   choose  a~covering of $A$ by $G$-balls $B(x_{jk}, \rho_{jk} )$, $j\in\nat$, with
$\rho_{jk}<1/k $ and $\sumj \rho_{jk}<\a_k+1/k$,   and define $A_k:=  \cupj B(x_{jk}, \rho_{jk})$.
Taking $A_k:=X$ otherwise, the set $\tilde A:=\bigcap_{k=1}^\infty A_k$ has the desired property.
 
Our main result is the following.
 
\begin{theorem}\label{main-result}
If  a set  $A$ in $X$ satisfies  $m_G(A)<\infty$, then it is contained in a~$G$-semipolar Borel set.
\end{theorem} 
 
In its proof we shall use the following simple generalization
of the Lebesgue dominated convergence theorem (straightforward
application of Fatou's lemma both to the sequence $(f_n)$ and the sequence 
$(g_n-f_n)1_{\{g_n<\infty\}}$).

\begin{lemma}\label{Leb-prime}
Let $(Y,\A, \lambda)$ be a measure space and let $f_n, g_n $ be  
$\A$-measurable functions on $Y$ such that $0\le f_n\le g_n$, the sequence 
 $(f_n)$ converges pointwise to~$f$, the sequence $(g_n)$ converges pointwise to~$g$,
and $\lim_{n\to\infty} \int g_n\,d\lambda=\int g\,d\lambda<\infty$. Then 
$\lim_{n\to \infty}\int f_n\,d\lambda=\int f\,d\lambda$.
\end{lemma}

\begin{proof}[Proof of Theorem \ref{main-result}]
By (\ref{aa-prime}), it suffices to consider a Borel set $A$ in $X$.

a) Suppose that $A$ is not $G$-semipolar. 
Then we may choose $\mu_0\in \M(X)$ such that $\mu_0(A)>0$,
$\mu_0(X\setminus A)=0$, and $G\mu_0$ is continuous and real.
Let $K$ be a compact in $A$ such that $\mu_0(K)>0$, and define $\mu:=1_K\mu_0$.
By (i),  $\limsup_{x\to y} G(x,y)=\infty$,  and hence 
\begin{equation}\label{point-zero}
 \mu(\{y\})=0\quad\mbox{  for every }y\in X.
\end{equation} 

b) Now let us fix a decreasing continuous function~$\vp$ on $[0,\infty]$ such that
$\vp=1$ on~$[0,1]$ and $\vp=0$ on~$[2,\infty]$. For $x\in X$ and $\rho>0$, let 
\begin{equation*}
        \mu_{x,\rho}:=\vp\left(\frac 1{\rho G(x,\cdot)}\right) \mu
\und p_\rho(x):=G\mu_{x,\rho}(x).
\end{equation*} 
Then
\begin{equation}\label{B-mu-B} 
       1_{B(x,\rho)}\mu\le \mu_{x,\rho}\le 1_{B(x,2\rho)}\mu \und
 p_\rho(x)\le p_\sigma(x)\mbox{ if }\rho\le \sigma.
\end{equation} 
Since $\mu(\{x\})=0$, by (\ref{point-zero}), and since  $\bigcap_{\rho>0} B(x,2\rho)$ is either empty or $\{x\}$, 
we see, in particular, that $p_\rho(x)\downarrow 0$ as $\rho\downarrow 0$.

Next we claim that all functions 
$p_\rho$ are continuous. To that end let us fix $\rho>0$,
$x_0\in X$, and $\ve>0$. Since $\mu(\{x_0\})=0$, we may choose a continuous 
function $0\le \psi\le 1$ which equals~$1$ on a neighborhood $V_0$ of $x_0$
and satisfies $G(\psi\mu)(x_0)<\ve$. Then $G(\psi\mu)$  is continuous and real.
So there exists a compact neighborhood $K_0$ of~$x_0$ in~$V_0$ such that 
$G(\psi\mu)<\ve$ on $K_0$. 
For every $x\in K_0$, let 
\begin{equation*}
  \vp_x:=(1-\psi)\vp\left(\frac 1{\rho G(x,\cdot)}\right), \qquad
  \mu_x:=\vp_x\mu, \qquad q(x):=G\mu_x(x).
\end{equation*} 
Then $0\le \vp_x\le 1$ and, for every $x\in K_0$, 
\begin{equation}\label{ve}
 p_\rho(x)-q(x)= G(\psi\mu_{x,\rho})(x)\le G(\psi\mu)(x)<\ve.
\end{equation}
To finish the proof of our claim let us finally  fix a sequence $(x_n)$ in $K_0$
converging to $x_0$. If $y\in X\setminus V_0$, then
\begin{equation*}
        \lim_{n\to \infty} G(x_n,y)=G(x_0,y) \und
        \lim_{n\to \infty} \vp_{x_n}(y)=\vp_{x_0}(y).
\end{equation*} 
If $y\in V_0$, then $\vp_{x_n}(y)=0=\vp_{x_0}(y)$ for every $n\in\nat$.
Hence the sequence $(\vp_{x_n}G(x_n,\cdot))$ converges pointwise
to   $\vp_{x_0}G(x_0,\cdot)$. Since $\lim_{n\to\infty} G\mu(x_n)=G\mu(x_0)$
and $G\mu(x_0)$ is finite, 
we thus conclude, by Lemma \ref{Leb-prime},  that 
\begin{equation*} 
 \lim_{n\to\infty} q(x_n)=q(x_0).
\end{equation*} 
 If $n$ is large enough, then $|q(x_n)-q(x_0)|<\ve$,
and   $|p_\rho(x_n)-p_\rho(x_0)|<3\ve$, by (\ref{ve}).

c) By (iii), there exist a compact neighborhood $L$ of $K$ and $a>0$ such that 
$G\le a$ on $(X\setminus L)\times K$. By Dini's theorem, there exists
a~decreasing sequence $(\tau_n)$  such that $0<\tau_n<1/n$ and
\begin{equation}\label{dini}
     G\mu_{x, \tau_n}(x)=  p_{\tau_n}(x)<2^{-n}, \qquad  x\in L,
\end{equation} 
for every $n\in\nat$.
For the moment, let us fix $k\in\nat$, $x\in L$ , 
$0<\rho<\tau_k$ and define
\begin{equation*}
\nu:=\sum_{n=1}^k \mu_{x,\tau_n}.
\end{equation*} 
By  (\ref{nu-B-G}) and (\ref{dini}), 
 $\nu(B(x,\rho))\le \rho G\nu(x)<\rho$. 
By  (\ref{B-mu-B}),  $\mu_{x,\tau_n}(B(x,\tau_n))=\mu(B(x,\tau_n))$
for every $1\le n\le k$.
Thus
\begin{equation}\label{ess}
               k \mu(B(x,\rho))= \sum_{n=1}^k \mu_{x,\tau_n} (B(x,\rho))
        =\nu(B(x,\rho))  <\rho.
\end{equation}

d) Again let $k\in\nat$ and let us fix $0<\delta<\min\{\tau_k, a\inv\}$.
Let us consider $x_j\in X$ and $0<\rho_j<\delta$, $j\in\nat$,  such that the sets $B(x_j,\rho_j)$ cover $K$. 
If $j\in\nat$ such that $x_j\in X\setminus L$, then $G(x_j,y)\le a<\delta\inv <\rho_j\inv $
for all $y\in K$, and thus $B(x_j,\rho_j)\cap K=\emptyset$. So we may assume  that $x_j\in L$ for every $j\in\nat$. 
Then, by (\ref{ess}),
\begin{equation*}
              k\mu(K)\le k\sum_{j=1}^\infty \mu(B(x_j,\rho_j))
<\sum_{j=1}^\infty \rho_j.
\end{equation*} 
Thus $m_G^{\delta}(K)\ge k\mu(K)$. This shows that  $m_G(K)=\infty$, and hence $m_G(A)=\infty$.
\end{proof} 

To deal with $G$-polar sets, we define, for sets $A$ in $X$, 
\begin{equation*}
        c(A):=\sup\{\nu^\ast(A)\colon \mbox{$\nu\in\M(X)$, $G\nu\le 1$}\}
\end{equation*} 
($\nu^\ast(A)$ denoting the outer measure of $A$). If $A$ is universally measurable, then 
\begin{equation*}
         c(A):=\sup\{\nu(A)\colon \mbox{$\nu\in\M(X)$, 
              $ \nu(X\setminus A)=0$, $G\nu\le 1$}\}
\end{equation*} 
(it suffices to note that whenever $\nu\in\M(X)$ such that $G\nu\le 1$,
then $\nu':=1_A\nu$ satisfies $G\nu'\le 1$ and $\nu'(A) =\nu(A)$).
In particular, a Borel set $A$ in $X$ is $G$-polar if and only if $c(A)=0$.

\begin{proposition}\label{cap-mG}
For every set $A$ in $X$, $c   (A)\le m_G(A)$.
\end{proposition}

\begin{proof} Let $\nu\in \M(X)$ such that $G\nu\le 1$, and let $x_j\in X$
and $\rho_j>0$, $j\in\nat$,  such that $A\subset \cupj B(x_j,\rho_j)$. Then, by (\ref{nu-B-G}),
\begin{equation*} 
  \nu^\ast(A)\le \sumj \nu(B(x_j,\rho_j))\le \sumj \rho_j G\nu(x_j)\le \sumj \rho_j.
\end{equation*} 
This clearly implies that $c  (A)\le m_G(A)$.
\end{proof}

\begin{corollary}\label{m-polar} 
If  a set  $A$ in $X$ satisfies  $m_G(A)=0$, then it is contained in a~$G$-polar Borel set.
\end{corollary} 

\section{Application to harmonic spaces and balayage spaces}\label{habasp}

Let $(X,\H)$ be a $\mathcal P$-harmonic space or, more generally, let $(X,\W)$ be a balayage space; 
 see \cite{bauer66, BH,H-course,Const}. 
Let us assume that the constant function $1$ is superharmonic and that we have a~(Green) function
$G\colon X\times X\to [0,\infty]$ having the following properties (note that (i$'$) implies (i) from Section \ref{G-section}):
\begin{itemize}
\item[\rm (i$'$)] For every $y\in X$, the function $G(\cdot,y)$ is a potential on $X$
which has  superharmonic support $\{y\}$ and satisfies $\limsup_{x\to y} G(x,y)=\infty$.
\item[\rm (ii)] $G$ is continuous outside the diagonal $\Delta$ and Borel measurable on $\Delta$.
\item[\rm (iii)] For every compact $K$ in $X$, there exists a compact   $L$ in $X$
such that $G$ is bounded on $(X\setminus L)\times K$.
\item[\rm (iv)] For every continuous real potential $p$ on $X$ having compact 
superharmonic support there exists a~measure  $\mu\in \M(X)$  such that 
\begin{equation*}
p=G\mu:=\int G(\cdot,y)\,d\mu(y).
\end{equation*} 
\end{itemize} 

In the case of a harmonic space $(X,\H)$, property (iii) follows from (i$'$) and (ii).
Indeed, let $K$ be a compact in $X$ and let $L$ be any compact neighborhood of $X$.
Then $G$ is bounded on $\partial L\times K$ by some $a>0$, and hence 
$G\le a$ on $(X\setminus L)\times K$, by the minimum principle; see \cite[III.6.6]{BH}.

Let us recall that, for all subsets $A$ of $X$ and superharmonic functions $u\ge 0$,
the reduced function $R_u^A$,  defined to be the  infimum of all positive superharmonic functions majorizing $u$
on $A$, is harmonic on $X\setminus \ov A$, 
 and that its greatest lower semicontinuous minorant $\hat R_u^A$ (the balayage of $u$ on $A$) 
is superharmonic on $X$. 

A subset $A$ of $X$ is called \emph{polar} if $\hat R_1^A=0$ (or, equivalently, if $\hat R_u^A=0$ for every superharmonic 
function $u\ge 0$ on $X$). It is called \emph{semipolar} if it is a countable union of sets $T_j$, $j\in\nat$, 
which are \emph{totally thin}, that is, such that, for every $x\in X$, there exists a superharmonic function $u\ge 0$
on $X$ with $\hat R_u^{T_j}(x)<u(x)$. Of course, every polar set is totally thin and hence semipolar.
Moreover, it is immediately seen that countable unions of polar (semipolar) sets are polar (semipolar)
and that  every subset of a polar (totally thin, semipolar) set is polar
(totally thin, semipolar). 

\begin{proposition}\label{G-semipolar-semipolar}
Let $A$ be a Borel set in $X$. Then $A$ is semipolar if and only if $A$ is $G$-semipolar.
\end{proposition}

\begin{proof} Immediate consequence of (iv) and \cite[VI.8.8]{BH} (using Remark \ref{G-semipolar},2). 
\end{proof} 

\begin{corollary}\label{m-semipolar}
If $A\subset X$ satisfies $m_G(A)<\infty$, then $A$ is semipolar.
\end{corollary} 

\begin{proof} Theorem \ref{main-result} and Proposition \ref{G-semipolar-semipolar}. 
\end{proof} 

To get a corresponding result for polar sets let us 
suppose from now on in this section that the following holds:
\begin{itemize}
\item[\rm ($\ast$)]
For every compact $K$ in $X$, the potential $\hat R_1^K$ is harmonic on $X\setminus K$.
\end{itemize} 

We observe that ($\ast$) does not hold in general; see \cite[V.9.1]{BH} for an example of a~balayage space
and a~compact $K$ such that  $\hat R_1^K$ is not harmonic on $X\setminus K$.

However, ($\ast$) is satisfied in the case of a harmonic space, since 
$R_1^K$ is harmonic  on~$X\setminus K$, and hence $\hat R_1^K= R_1^K$ on $X\setminus K$. 
It also holds in the case of a balayage space, where
every semipolar set is polar (since then the semipolar sets $\{\hat R_1^K<R_1^K\}$ are polar and  harmonic measures 
do not charge polar sets; see \cite[VI.5.11 and VI.5.6]{BH}).

Moreover, we assume that $G$ has the following property:
\begin{itemize}
\item [\rm (iv$'$)] For every bounded potential $p$ with compact superharmonic support
there exists a measure $\mu\in\M(X)$ such that $p=G\mu$.
\end{itemize} 

Clearly,  (iv$'$) implies  (iv),  since,  by the minimum principle,
every continuous real potential~$p$ with compact superharmonic support $K$
is bounded by its maximum on~$K$; see \cite[III.6.6]{BH}. 

Conversely,  by \cite[Theorem 4.1]{HN-representation},   (iv$'$) is a consequence of (iv)
if, for every $x\in X$,  the function $G(x,\cdot)$ is lower semicontinuous at $x$
 and, if $x$ is finely isolated (and not isolated), continuous at $x$.  

\begin{proposition}\label{G-polar-polar}
Every $G$-polar Borel set in $X$ is polar. 
\end{proposition}

\begin{proof} Let $A$ be a Borel set in $X$ which is not polar.
  By \cite[VI.5.5]{BH}, there exists a~compact $K$ in $A$ such that $K$ is not polar,
that is, $\hat R_1^K\ne 0$. Of course, $\hat R_1^K\le 1$.  By (iv$'$),  $\hat R_1^K=G\mu$
for some $\mu\in \M(X)$. By ($\ast$), $\hat R_1^K$ is harmonic on $X\setminus K$,
and hence $\mu(X\setminus K)=0$, by (i$'$). Therefore $A$ is not $G$-polar.
\end{proof}

\begin{corollary}\label{m0polar}
If $A\subset X$ satisfies $m_G(A)=0$, then $A$ is polar. 
\end{corollary} 

\begin{proof}
Corollary \ref{m-polar} and  Proposition \ref{G-polar-polar}.
\end{proof}

\section{Application to standard examples}\label{standard} 
\subsection{Classical potential theory and Riesz potentials}

  Let $X=\realn$,   $0<\b<n$, $0<\b\le 2$. If $\b=2$ (and $n\ge 3$),  
let $(\realn,\H)$ be the
harmonic space associated with the Laplace operator. If $\b<2$, let $(\realn, \W)$ be the 
balayage space, where every superharmonic function is an increasing limit of Riesz potentials;
see \cite[Section V.4]{BH}. In both cases, the function $G$ defined by
\begin{equation*} 
                            G(x,y):=|x-y|^{\b-n}, \qquad x,y\in\realn,
\end{equation*} 
has the properties (i$'$), (ii), (iii) and (iv$'$) and every semipolar
set is polar. Of course,
 for all $x\in \realn$ and $\rho>0$,
\begin{equation*}
                   \{y\in \realn \colon G(x,y)>\rho\inv\}=\{y\in \realn\colon |x-y| < \rho^{n-\b}\}.
\end{equation*} 
Hence $m_G$ is the usual $(n-\b)$-dimensional Hausdorff measure $m_{t^{n-\b}}$ on $\realn$.
Therefore Corollary \ref{m-semipolar}  yields the following,

\begin{theorem}\label{class-riesz}
If a set $A$ in $\realn$ satisfies  $m_{t^{n-\b}}(A)<\infty$, then $A$ is polar.
\end{theorem}

Applying  Corollary \ref{m-semipolar} to classical Green functions $G_D$ on  discs $D$ in $\real^2$,
we obtain the following.

\begin{theorem}\label{class-plane}
Let $A\subset \real^2$ and $\phi(t):=\log^+(1/ t)$, $t>0$.  If $m_\phi(A)<\infty$, then~$A$ is polar
{\rm(}with respect to classical potential theory{\rm)}. 
\end{theorem} 

In particular, we have recovered  known results dealing with classical or Riesz potential
theory; see \cite[Theorem 5.9.4]{gardiner-armitage}, 
\cite[Theorem IV.1]{carleson}, \cite[Theorem 3.14]{landkof}.

\subsection{Heat equation}

In this subsection we consider the harmonic space $(\real^{n+1}, \H)$ 
associated with the heat equation
with Green function $G$ defined by (\ref{def-W}) and (\ref{G-heat}). Then $G$ obviously has  the  properties 
(i) -- (iii)  from Section \ref{G-section}, and property (iv$'$) follows from \cite[Corollary~6.39 and Theorem 6.34]{watson-book}. 

Let $x'\in \real^{n+1}$, $\rho>0$.  Then $ B(x',\rho)=x'+B(0,\rho)$, where 
 $B(0,\rho)$ denotes the \emph{heat ball 
 with center $0$ and radius $ \rho^{2/n}$}, that is,
\begin{equation}\label{B-form}
   B(0,\rho):=  \left\{(y,-s)\colon y\in \realn,\,   0<s<\rho^{2/n},\,
       |y|^2<  2ns \log \frac {\rho^{2/n}}  s\right\};
\end{equation} 
in particular,  $B(0,\rho)$ is convex and contained in the  cylinder
\begin{equation*} 
              \bigl\{(y,-s)\colon y\in \realn,\, 0<s< \rho^{2/n}, \, |y|< n^{1/2} \rho^{1/n} \bigr\};          
\end{equation*} 
see \cite[page 2]{watson-book}. 
Hence, by our definition  (\ref{def-P}),
\begin{equation}\label{B-P}
  B(x',\rho)\subset  P(x', 2n^{1/2} \rho^{1/n} ).
\end{equation} 
Further, we claim that, defining $z':=(0,\rho^2)$, 
\begin{equation}\label{suff-P-B}
  P(-z',\rho)\subset B(0, 2^{n/2}\rho^n)
\end{equation} 
and hence, by translation invariance,
\begin{equation}\label{P-B}
P(x',\rho)\subset B(x'+z', 2^{n/2}\rho^n).
\end{equation} 
 To prove (\ref{suff-P-B}) we only have to show
that every vertex of $P(-z',\rho)$ is
contained in the convex set $ B(0,2^{n/2}\rho^n)$ which, by (\ref{B-form}),
is the set of all $(y,-s)\in \real^{n+1}$ such that 
\begin{equation}\label{B-est}
 0<s<2\rho^2 \und 
  |y|^2<2ns\log\frac{2\rho^2}s.
\end{equation} 
Such a vertex has the form $(y, -s)$ with $s=(k/2)\rho^2 $,
$k\in\{1,3\}$ and $|y|^2=n(\rho/2)^2$, hence it is contained in $ B(0,2^{n/2}\rho^n)$
if
\begin{equation*}
           \frac 14 < 2 \frac k2 \log \frac 4k, \qquad k\in\{1,3\}.
\end{equation*} 
This holds, of course, for $k=1$. It also holds for $k=3$ since already
$4\log (4/3)=\log ( 256/81)>$1. So (\ref{P-B}) holds.

\begin{proposition}\label{mGnP}
For every subset $A$ of $\real^{n+1} $,
\begin{equation*} 
m_{n,P}(A)\le 2^n n^{n/2} m_G(A) \und           m_G(A)\le 2^{n/2} m_{n,P}(A). 
\end{equation*} 
\end{proposition} 

\begin{proof}  By our definitions, $m_G=m_{1,B}$; see Section \ref{sec-comp}.
Hence the first
  inequality follows from  (\ref{B-P}) and 
 Proposition \ref{general-comparison} 
  with $F=B$ and  $\eta=1$,  $\tilde F=P$ and $\tilde \eta=n$, $\kappa=
  2n^{1/2}$. 

The second inequality follows from  (\ref{P-B}) and 
 Proposition \ref{general-comparison} 
  with $ F=P$ and $ \eta=n$, $\tilde F=B$ and  $\tilde \eta=1$,
  $\kappa=2^{n/2}  $. 
\end{proof}

Let us next observe that, for $m_P$ from Section \ref{intro} and $m_{n,P}$ we have
\begin{equation}\label{PnP}
                   m_P=n^{n/2} m_{n,P}.
\end{equation} 
Indeed, if $\delta>0$, then 
$
                 \sqrt n \rho\le \diam P(x',\rho)\le \sqrt n \rho(1+\delta)
$
for all $x'\in \real^{n+1}$ and $0<\rho<\delta$, and hence 
\begin{equation*}
                              n^{n/2} m_{n,P}^\delta\le m_P^\delta\le
                              (1+\delta)^n n^{n/2} m_{n,P}^\delta,
\end{equation*} 
which clearly implies (\ref{PnP}). So we may state Proposition
\ref{mGnP} in the following way.

\begin{corollary}\label{PGP}
 For every subset $A$ of $\real^{n+1} $,
\begin{equation*} 
             (2n)^{-n}  m_P(A)\le m_G(A)\le (2/n)^{n/2} m_P(A). 
\end{equation*} 
In particular, $(1/2) m_P\le m_G\le \sqrt2\, m_P$, if $n=1$.
\end{corollary} 

So Corollaries  \ref{m-semipolar} and   \ref{m0polar}
yield the following, where the first statement is the positive answer to the question raised in \cite[page 330]{TW}
and the second one is \cite[Theorem 1]{TW}.   

\begin{theorem}\label{heat-theorem}
Let $A$ be a subset of $X$.
\begin{itemize}
\item [\rm 1.] If  $m_P(A)<\infty$, then $A$ is semipolar.
\item [\rm 2.] If  $m_P(A)=0$, then $A$ is polar.
\end{itemize} 
\end{theorem} 

Let us note that an analogous result to that in 
Theorem \ref{heat-theorem}.2 was established in       
\cite[Theorem 3]{mysovskikh}, where a modified heat kernel is investigated. 

Now let us suppose that $n=1$ and show that subsets of vertical lines in $\real^2$
are polar if they are semipolar (as we already noticed, this is not true for subsets 
of horizontal lines). 

\begin{proposition}\label{L}
Every semipolar set $A$ in $ \{0\}\times \real$ is polar.
\end{proposition} 

\begin{proof} Let $A$ be a semipolar set in $\{0\}\times \real$.
By \cite[VI.5.7.3]{BH}, we may suppose that~$A$ is a~Borel set.
We claim that every compact $K$ in $A$ is polar, and hence $A$
is polar, by~\cite[VI.5.5]{BH}.
 
To that end suppose that $A$ contains a compact $K$ which is not polar,
and let~$\nu$ be the thermal capacitary distribution of $K$
; see \cite[Definition 7.33]{watson-book}.
Then $\nu\ne 0$,  $\nu$ is supported by $K$, and $G\nu\le 1$. By
Lusin's theorem, there exists a compact~$\tilde K$ in~$K$ such that 
$\tilde \nu:=1_{\tilde K}\nu\ne 0$ and the restriction of $G\nu$ to $\tilde K$ is continuous.
Since both~$G\tilde \nu$ and~$G(\nu-\tilde \nu)$ are lower semicontinuous, 
we see that also the restriction of~$G\tilde \nu$ to~$\tilde K$ is continuous.
By \cite[Theorem 5]{Dont}, this implies that  $G\tilde \nu$
is continuous on~$\real^2$. So $\tilde K$ is not semipolar, a contradiction.
\end{proof}

Let us remark that a set in $\{0\}\times \real$ is polar if and only if it has zero Riesz
$\frac 12$-capacity; see \cite[Theorem 2]{KW}. 

\section{Application to  space-time processes}\label{space-time}

In this section we shall  consider a space-time setting which is more general 
than for the heat equation and which will be discussed in the situation of heat semigroups
on metric measure spaces in Section \ref{section-metric}.

As in \cite[Sections 7 and 8]{dense-para} we assume the following.
Let  $X\ne\emptyset$  
be a locally compact space with countable base, 
$X':=X\times \real$, and let $\B^+(X)$, $\B^+(X')$ denote the set of all Borel measurable positive numerical functions
on $X$, $X'$ respectively.
We suppose that we have a~measure~$\mu$ on~$X$, not charging points,
and a~strictly positive continuous real function 
$(x,y,t)\mapsto p_t(x,y)$ on 
$X\times X\times (0,\infty)$ satisfying the Chapman-Kolmogorov equations:
\begin{itemize} 
\item[\rm(CK)]
 For all $s,t\in (0,\infty)$ and $ x,y\in X$,
$$ 
 p_{s+t}(x,y)=\int p_{s}(x,z)p_{t}(z,y)\,d\mu(z).
$$
\end{itemize} 
For $t>0$, $f\in \B^+(X)$ and $x\in X$, let 
\begin{equation*}  
  P_tf(x) := \int p_t(x,y)f(y)\,d\mu(y). 
\end{equation*} 
Then $\mathbbm P:=(P_t)_{t>0}$ 
 is a  semigroup on $X$. 
Let   $ E_{\mathbbm P}$
denote the corresponding cone of excessive functions, that is,
 \begin{equation*}
  E_{\mathbbm P}:=\{u\in \B^+(X)\colon \sup_{t>0}P_tu=u\}. 
 \end{equation*} 

We suppose that, in addition, the following holds.
\begin{itemize}
\item[\rm(E)] 
$1\in E_{\mathbbm P}$. 
\item[\rm(C)]
For all $x_0,y_0\in X$,  
\begin{equation*}
\limsup_{(x,t)\to (x_0,0)} p_t(x, x_0) =\infty \und  \lim_{(x,y,t)\to (x_0,y_0,0)} p_t(x,y)=0, \quad\mbox{ if }x_0\ne y_0.
\end{equation*} 
\item[\rm(KL)] 
For all compacts $K$ in $X$, $T>0$ and $\ve>0$, there exists a compact $L$ in $X$
such that 
\begin{equation*} 
  \|1_{X\setminus L}\,  p_t(\cdot,y)\|_{L^1(\mu)} +  \|1_{X\setminus L}\,  p_t(\cdot,y)\|_\infty <\ve
  \quad\mbox{ for all }(y,t)\in K\times (0,T].
\end{equation*} 
\end{itemize} 

In \cite{dense-para} we were supposing that also the  dual function $(x,y,t)\mapsto p_t(y,x)$ 
has the properties above. However, since this played a role only starting with \cite[Lemma~8.5]{dense-para},
we shall not need that here.

Let $\mathbbm T:=(T_t)_{t>0}$ denote the 
semigroup of uniform translation to the left, that is,
$T_t(r,\cdot):=\ve_{r-t}$. We define
\begin{equation*} 
\mathbbm P':=\mathbbm P\otimes \mathbbm T
\end{equation*} 
which means that $\mathbbm P'=(P_t')_{t>0}$, where, for all $t>0$,  $f\in \B^+(X')$ and $(x,r)\in X'$,
\begin{equation*} 
       P_t'f(x,r):=P_tf(\cdot, r-t)(x).
\end{equation*} 
Clearly, $\mathbbm P'$ is a semigroup on $X'$. Let $E_{\mathbbm P'}$ denote its cone of excessive functions.
Then, by \cite[Proposition 8.1]{dense-para} and (E),  the following holds; cf.\ also \cite[Section V.5]{BH}.

\begin{proposition}\label{bal-space}
$(X',E_{\mathbbm P'})$ is a  balayage space, $1\in E_{\mathbbm P'}$. 
\end{proposition}

We define $G'\colon X'\times X'\to [0,\infty)$ by
\begin{equation*} 
 G'((x,r),(y,s)):=\begin{cases}  p_{r-s}(x,y), &\quad\mbox{ if } r>s,\\
                                                           0, &\quad\mbox{ if } r\le s.
                       \end{cases}
\end{equation*} 

\begin{theorem}\label{G-para}
The function $G'$ has the properties {\rm(i$'$), (ii), (iii)} and {\rm(iv$'$)} of Section \ref{habasp}. 
\end{theorem}

\begin{proof} By   (C), the function $G'$ is
continuous outside the diagonal in $X'\times X'$ and  lower
 semicontinuous on $X'\times X'$.
Hence $G'$ satisfies (ii). Moreover, (i$'$)  holds, by~\cite[Proposition 8.4]{dense-para}. 

Property (iii) is a consequence of   (KL), (CK) and (E). Indeed, let $K'$ be a compact in $X'$.
Then there exist  a compact $K$ in $X$, $t_0\in \real$, and $T>0$, such that $K'$ is contained in 
$ K\times [t_0, t_0+T]$. 
By (KL), there exists a compact  $L$ in $X$ such that
\begin{equation}\label{p-first}
p_t(x,y)\le 1\qquad\mbox{ for all } x\in X\setminus L,\, y\in
K \mbox{ and } t\in (0,T].
\end{equation} 
Let $a\ge 1$ be such that $p_T\le a$ on the compact $L\times K$.
By (\ref{p-first}), $p_T \le 1$ on~$(X\setminus L)\times K$.
  Therefore $p_T\le a$ on $X\times K$ and hence
\begin{equation}\label{p-second} 
         p_s(x,y)=\int p_{s-T}(x,z) p_T(z,y)\,d\mu(z)\le a  P_{s-T}1(x)\le a
\end{equation} 
for all $x\in X$, $y\in K$ and $s>T$.
Defining $L':=L\times [t_0,t_0+T]$ and combining (\ref{p-first}) and (\ref{p-second}), we obtain that 
$G'\le a$ on $(X'\setminus L')\times K'$. 

To see that $G'$ has the property (iv$'$) we recall that the potential kernel $V'$ of $\mathbbm P'$
is proper and, defining $m:=\mu\otimes \lambda_\real$,
\begin{eqnarray*} 
     V'f(x,r)&=&\intoi P_t'f(x,r)\,dt
       =\intoi P_t f(\cdot,r-t)(x)\,dt\\
&=& \intoi \bigl(\int_{X} p_t(x,y)f(y,r-t)\,d\mu(y)\bigr) \, dt\\
       &  =&  \int G'((x,r),(y,s))f(y,s)\,dm(y,s)
\end{eqnarray*}
for every $f\in\B^+(X')$; see \cite[page 674]{dense-para}. 

Let $q$ be a potential for the balayage space $(X',E_{\mathbbm P'})$.
By  \cite[II.3.11]{BH}, there exist bounded functions $f_n\in \B^+(X')$ such that $V'f_n\uparrow q$
 as $n\to \infty$.
Defining $\mu_n:=f_n m$ we hence know that $G'\mu_n\uparrow q$. Moreover, we note that 
 every balayage 
space contains only countably many finely isolated points; see \cite[III.7.2]{BH}.  Hence there is no
 point $(x,r)\in X\times \real$ which is finely isolated, since otherwise \emph{all} points $(x,t)$, $t\in \real$, 
would be finely isolated.
By \cite[Theorem~1.1]{HN-representation},
we finally conclude that there exists $\mu'\in \M(X')$ such that $q=G'\mu'$; cf.\ 
also \cite[Remark~8.7]{dense-para} and the references therein.
\end{proof} 

Since $(X', E_{\mathbbm P'})$ is a balayage space, there  exists a Hunt process 
 $\mathfrak X'$ on~$X'$ with transition semigroup $\mathbbm P'$, the 
\emph{space-time process} associated with $(x,y,t)\mapsto p_t(x,y)$  and $\mu$; see \cite[Theorem IV.8.1]{BH} and its proof. 
So, by Corollary~\ref{m-semipolar}, we have the following result.

\begin{corollary}\label{para-corollary}
If $A$ is a subset of $X'$ such that $m_{G'}(A)<\infty$, then there exists  
a~Borel set $\tilde A$ containing $A$ which is semipolar, that is, the process $\mathfrak X'$
  hits the set $\tilde A$ at most countably many times. 
\end{corollary}

\section{Application to heat kernels on metric measure spaces}\label{section-metric}

Let $(X,d)$ be a separable metric space, $X\ne\emptyset$, 
where  balls $D(x,r):=\{ d(\cdot,x)<r\}$ are relatively compact,
  let $\mu$ be a positive Radon measure on $X$ with full support, 
and suppose that we have  a~continuous positive real function 
\begin{equation}\label{intro-pt}
   (x,y,t)\mapsto p_t(x,y) \on X\times X\times (0,\infty)
\end{equation} 
such that the  following holds:
\begin{itemize} 
\item[\rm (CK)]
\emph{Chapman-Kolmogorov equations}:   
For all $s,t\in (0,\infty)$ and $ x,y\in X$,
\begin{equation*} 
 p_{s+t}(x,y)=\int p_{s}(x,z)p_{t}(z,y)\,d\mu(z). 
\end{equation*} 
\item[\rm (E)]
For every $x\in X$,
\begin{equation*}
                 \sup_{t>0} \int p_t(x,y)\,d\mu(y)= 1.
\end{equation*} 
\item[\rm (B)]
There exist   constants $\a,\b>0$ and positive decreasing real functions $\Phi_1$, $\Phi_2$ 
on~$[0,\infty)$ such that $\Phi_1(1)>0$, $   \int_0^\infty  \sigma^{\a-1} \Phi_2(\sigma) \,  d\sigma <\infty$
and
\begin{equation}\label{bounds}
\frac 1{t^{\a/\b}} \Phi_1\left(\frac{d(x,y)}{t^{1/\b}}\right) \le p_t(x,y) \le \frac 1{t^{\a/\b}} \Phi_2\left(\frac{d(x,y)}{t^{1/\b}}\right) 
\end{equation} 
 for all $x,y\in X$ and $t>0$.
\end{itemize}

\begin{remarks}
{\rm 
1. See \cite[(1.5) and (H$_0$) on p.\,2067]{ghl} for assumption (B)
and  \cite{grigor-telcs} as well as \cite[Theorem 2.10]{gcap} for conditions implying
the continuity of (\ref{intro-pt}). For the definition of an abstract heat kernel, the relation to Dirichlet
forms  and a~discussion of various examples see  \cite{grigor-hu-lau}. 
For manifolds, (E) follows from 
\cite[(7.50) and (7.53) in Theorem 7.13]{grigor-book}.

2.  A striking fact is the following dichotomy; see \cite[Theorem 4.1]{grigor-kumagai}.
Suppose that we have an abstract heat kernel satisfying  (\ref{bounds}) 
with functions of the form  $\Phi_j(\sigma)=C_j\Phi(c_j \sigma)$,  where $C_j,c_j\in (0,\infty)$
and $\Phi\colon [0,\infty)\to [0,\infty)$ is decreasing with $\Phi(\sigma_0)>0$ for some $\sigma_0>0$. Then, under mild
additional assumptions on $X$ and $p_t(x,y)$,    
either $\b\ge 2$ and (\ref{bounds}) holds with 
\begin{equation}\label{Gaussian}
                \Phi(\sigma)=\exp\bigl(-\sigma^{\frac \b{\b-1}}\bigr ),  
\end{equation}
leading to \emph{sub-Gaussian} lower and upper bounds, or  (\ref{bounds}) holds with
\begin{equation}\label{stable}
                 \Phi(\sigma)=(1+\sigma)^{-(\a+\b)}, 
\end{equation} 
leading to \emph{stable-like} lower and upper bounds.
In the case (\ref{Gaussian}) the corresponding process will be a diffusion, 
in the case (\ref{stable}) it will be a jump process.

3.  A special case for sub-Gaussian bounds is, of course, the classical Gauss-Weierstrass kernel on $X=\realn$
(with $d(x,y)=|x-y|$ and Lebesgue measure $\mu$), where $\a=n$, $\b=2$, $C_1=C_2=(4\pi)^{-n/2}$ and $c_1=c_2=1/4$.
See  also \cite{li-yau} and \cite{grigor-ussr-92} for two-sided Gaussian bounds (sub-Gaussian bounds
with $c_1=c_2$) for heat kernels on manifolds. 
Many further examples for sub-Gaussian bounds (with $\b>2$)  are given by fractal spaces like Sierpinski gaskets 
and carpets;  see  \cite{barlow} and \cite{barlow-bass-carpet}. For properties which
in the setting of regular local Dirichlet forms are equivalent to two-sided sub-Gaussian bounds
we refer the reader to \cite{grigor-telcs}, \cite{grigor-hu} and \cite{grigor-hu-lau-2}. 

4. A special case with stable-like bounds is the $\b$-stable heat kernel
given by  the fractional Laplacian $(-\Delta)^{\b/2}$ on $\realn$ with  $\a=n$ and $0<\b<\min\{2,n\}$. 
More generally, subordination applied to heat kernels with Gaussian bounds leads to heat kernels with
stable-like bounds; see \cite{stos} and \cite{grigor-heat-handbook} for further examples obtained by subordination. 
For a direct approach on $d$-sets see \cite{chen-kumagai-dsets}. Moreover, see \cite{grigor-hu-hu} for properties 
characterizing the existence of two-sided stable like bounds in the setting of regular Dirichlet forms having 
a~jumping part (but no killing part).
  }
\end{remarks}

Let us now verify  the properties (C) and (KL) introduced in the previous section. 
To that end we define 
\begin{equation}\label{def-vp2}
      \vp_2(\sigma):=\sigma^\a\Phi_2(\sigma), \qquad \sigma>0,
\end{equation} 
and observe that taking  $\sigma:=d(x,y)t^{-1/\b}$
the upper bound $q_t(x,y)$ in (\ref{bounds}) can be written as
 \begin{equation}\label{qt}
                                  q_t(x,y)=t^{-\a/\b} \Phi_2(\sigma)=d(x,y)^{-\a} \vp_2(\sigma).
 \end{equation} 
Moreover, we note that $\lim_{\sigma\to\infty} \vp_2(\sigma)=0$ 
and hence, in particular,
\begin{equation*} 
M:=\sup\{\vp_2(\sigma)\colon 0\le \sigma<\infty\}< \infty.
\end{equation*} 
Indeed, let $I_k:=[2^k,2^{k+1}]$ and $\g_k:=\inf \vp_2(I_k)$, $k\in\nat$. 
Clearly,
$\int_{I_k} \vp_2(\sigma) \sigma\inv\, d\sigma \ge \g_k/2 $ for every $k\in\nat$,
and hence the integrability of $\sigma^{\a-1}\Phi_2$  implies that $\lim_{k\to\infty}\g_k=0$. 
Since $\Phi_2$ is decreasing 
and $\sigma^\a\le 4^\a \tau^\a$ for all 
$\sigma,\tau\in I_k\cup I_{k+1}$, we obtain that $\vp_2\le 4^\a \g_k$ on $I_{k+1}$, $k\in\nat$. 
Thus $\lim_{\sigma\to\infty} \vp_2(\sigma)=0$. 

\begin{lemma}\label{C}
Property {\rm(C)} of Section \ref{space-time} holds.
\end{lemma}

\begin{proof}  The lower estimate in (\ref{bounds})   
yields  that $\lim_{t\to 0}   p_t(x_0,x_0)=\infty$ for all $x_0\in X$. 
Since $\lim_{\sigma\to\infty} \vp_2(\sigma)=0$, (\ref{bounds}) and  (\ref{qt}) imply that  $\lim_{(x, y,t)\to (x_0,y_0,0)}
p_t(x,y)=0$, whenever $x_0,y_0$ are different points in $X$. 
\end{proof} 

\begin{lemma}\label{basic}
Let $K$ be a compact in $X$ and $\ve>0$. 
Then there exist  $T>0$ and a~compact~$L$ in~$X$ 
such that, for   $t>0$, $x\in X$ and $y\in K$,
\begin{equation}\label{KL-first}
p_t(x,y) 
<\ve, \quad\mbox{ whenever } 
t>T \mbox{ or } x\in X\setminus L. 
\end{equation} 
\end{lemma} 

\begin{proof} 
There exist $y_0\in
X$ and  $R>0$ such that $K\subset
D(y_0,R)$. Let us choose $T,N\in(0,\infty)$ such that
$T^{-\a/\b}\Phi_2(0)<\ve$ and $N^{-\a}M<\ve $. 
Let $L$ be the closure of~$D(x_0,R+N)$ 
and  $t>0$, $x\in X$, $y\in K$.   

If $t>T$, then $p_t(x,y)<\ve$, by (\ref{bounds}) and our choice of $T$.
If $x\in X\setminus L$, then $d(x,y)>N$, and hence $  p_t(x,y)<\ve$,
by (\ref{qt}) and our choice of $N$. 
\end{proof}

Having (E) and the lower bound in (B), 
we obtain that, taking
  $c_\mu:=\Phi_1(1)\inv$,  
\begin{equation}\label{volume}
                \mu(D(x,r))\le c_\mu r^\a\quad\mbox{ for all } x\in X \mbox{ and } r>0.
\end{equation} 
In particular, $\mu$ does not charge points and no point in $X$ is isolated.
Indeed,  it suffices to observe that, by the first inequality in (\ref{bounds}),
\begin{equation*} 
1\ge  \int p_{r^\b}(x,y)\,d\mu(y)\ge  \mu(D(x,r)) \inf_{y\in D(x,r)} p_{r^\b}(x,y)
\ge \mu(D(x,r))  r^{-\a} \Phi_1(1) 
\end{equation*} 
(cf.~the first lines of the proof for \cite[Theorem 3.2]{ghl}).  

\begin{lemma}\label{KL-holds}
Property {\rm (KL)} of Section \ref{space-time} holds. 
\end{lemma} 

\begin{proof}  Let $K$ be a compact in $X$, $T>0$ and $\ve>0$. We fix $y_0\in X$ and $R\ge 1$ such that 
$K\subset D(y_0,R)$ and define
\begin{equation*} 
D_j:=D(y_0, 2^j R ), \qquad j\in\nat.
\end{equation*} 
Since $\sigma^{\a-1}\Phi_2$ is integrable, there exists $k\in\nat$ such that, 
defining $\sigma_k:=2^{k-2}R/T^{1/\b}$, 
\begin{equation}\label{def-ve}
         2^{3\a+1} c_\mu       \int_{\sigma_k}^\infty\vp_2(\sigma)\,\frac{d\sigma}\sigma <\ve/2.
\end{equation} 
By Lemma \ref{basic}, we may assume that $p_t(x,y)<\ve/2$ whenever $y\in K$, $x\in X\setminus D_k$
and $t>0$ .
We claim that the closure $L$ of $D_k$ has the desired properties.

So let us fix $0<t\le T$ and $y\in K$.  If $j\in\nat$ and $x\in D_{j+1}\setminus D_j$, 
then $d(x,y)\ge 2^{j-1}R$, and hence, by the monotonicity of $\Phi_2$ and (\ref{volume}),
\begin{equation*}
\int_{D_{j+1}\setminus D_j} p_t(x,y)\,d\mu(x)\le  \frac 1{t^{\a/\b}} \Phi_2\left(\frac{2^{j-1}R}{t^{1/\b}}\right) \cdot
           c_\mu   \bigl(2^{j+1}R\bigr)^\a= 4^\a c_\mu\vp_2 \left(\frac{2^{j-1}R}{t^{1/\b}}\right).
\end{equation*}  
If $b\in (0,\infty)$, then 
$
                         \vp_2(2b) =(2b)^\a \Phi_2(2b)\le 2^{\a+1}\sigma^\a\Phi_2(\sigma) \cdot  ( b/\sigma)
$
for every $\sigma\in[b,2b]$, 
and therefore, integrating on $[b,2b]$,
\begin{equation*}
                 \vp_2(2b) \le 2^{\a+1} \int_b^{2b} \vp_2(\sigma)\,\frac{d\sigma}\sigma.
\end{equation*} 
Since $2^{k-2}R/t^{1/\b}\ge \sigma_k$, we  conclude that
\begin{equation*}
    \int_{X\setminus D_k} p_t(x,y)\,d\mu(x)\le 4^{\a}c_\mu\sum_{j=k}^\infty \vp_2\left(\frac{2^{j-1}R}{t^{1/\b}}\right)
  \le 2^{3\a+1}c_\mu\int_{\sigma_k}^\infty \vp_2(\sigma)\,\frac{d\sigma}\sigma<\ve/2. 
\end{equation*} 
Thus (KL) holds. 
\end{proof}

\subsection{Semipolar sets in $X'=X\times \real$}

Having verified the properties (C) and (KL)  
we may apply the results from  Section~\ref{space-time}.
Let us first recall the   definitions 
\begin{equation*} 
  P_t'f(x,r):=\int p_t(x,y)f(y,r-t)\,d\mu(y), \qquad\mbox{$t>0$, $(x,r) \in X'$, $f\in \B^+(X')$},
\end{equation*} 
and
\begin{equation*}
      G'((x,r),(y,s)):=\begin{cases}
             p_{r-s}(x,y),&\mbox{\quad if } r>s,\\
         \ \ \  \   0      ,&\mbox{\quad if } r\le s, 
                     \end{cases}
\qquad \mbox{ $ (x,r), (y,s)\in X'$}. 
\end{equation*} 
Then  we know the following; see Proposition \ref{bal-space}, Theorem \ref{G-para} and Corollary \ref{para-corollary}.

\begin{theorem}\label{all-para}
\begin{itemize}
\item[\rm 1.] $\mathbbm P'=(P_t')_{t>0}$  is a~sub-Markov semigroup
on $X'$ such that $(X',E_{\mathbbm P'})$ is a balayage space with $1\in E_{\mathbbm P'}$.
\item[\rm 2.] There exists a Hunt process $\mathfrak X'$ on $X'$ with transition semigroup $\mathbbm P'$.
\item[\rm 3.] $G'$ is a   Green function for $(X', E_{\mathbbm P'})$ satisfying
{\rm (i$'$), (ii), (iii), (iv$'$)} of Section  \ref{habasp}.
\item [\rm 4.] Every set $A$ in $X'$ with
$m_{G'}(A)<\infty$ is contained in a~Borel set which is semipolar, that is,   which the process $\mathfrak X'$
hits at most countably many times.
\end{itemize} 
\end{theorem} 

Moreover, we shall see that $m_{G'}$ 
 is equivalent to an anisotropic Hausdorff measure~$m_{\a,\b}$.
 To that end we first recall that in our setting, for   $x'\in X'$ and $\rho>0$,
\begin{equation*} 
                         B(x',\rho)=\{y'\in X'\colon G'(x',y')>1/\rho\}
\end{equation*} 
and that $m_{G'}=m_{1,B}$; see Section \ref{sec-comp}. 
That is, for every set $A$ in $X'$,  
\begin{equation*}
         m_{G'}(A)=\lim_{\delta\to 0}
                     \inf\big\{\sumj  \rho_j\colon 
                      A\subset \bigcup_{j=1}^\infty B(x_j',\rho_j),\ x'_j\in X', \, 0<\rho_j<\delta\big\}.
\end{equation*} 
 We now consider cylinders $\zb(x',\rho)$ in $X'$ 
given,  for $x'=(x,r)\in X'$ and $\rho>0$, by
 \begin{equation*}
    \zb(x',\rho):=\bigl\{(y,s)\in X'\colon d(x,y)<\rho\mbox{ and }
    |r-s|<\rho^\b\bigr\} 
 \end{equation*} 
 and define, for subsets $A$ of $X'$ and $\delta>0$,
 \begin{equation*} 
  m_{\a,\b}^\delta(A):=\inf\biggl\{\sum_{j=1}^\infty \rho_j^\a\colon
  A\subset \bigcup_{j=1}^\infty \zb(x_j',\rho_j),\, x_j'\in
  X',\, 0<\rho_j<\delta\biggr\},
 \end{equation*},
 \begin{equation*}
    m_{\a,\b}(A):=\lim_{\delta\to 0} m_{\a,\b}^\delta(A).
 \end{equation*}

\begin{proposition}\label{BZ}
There exists $C>0$ such that, for all
  $x'\in X'$ and $\rho>0$,  
\begin{equation}\label{BssZ}
B(x', \rho)\subset Z_\b(x',(C\rho)^{1/\a}).
\end{equation} 
In particular, $m_{\a,\b}\le C m_{G'}$. 
\end{proposition} 

\begin{proof}  Let $x'=(x,r)\in X'$, $\rho>0$ and $y'=(y,s)$ in $B(x',\rho)$.  
Then $t=r-s>0$ and, taking $\sigma:=d(x,y)t^{-1/\b} $,
 \begin{equation*}
              1/\rho< {G'}(x',y')=p_t(x,y)\le  t^{-\a/\b} \Phi_2(\sigma)= d(x,y)^{-\a}\vp_2(\sigma).
\end{equation*} 
So $t<(\Phi_2(0)\rho)^{\b/\a}$ and $d(x,y)<(M\rho)^{1/\a}$.  That is, (\ref{BssZ}) holds
with $C:=\max\{\Phi_2(0),  M\}$. An application of Proposition \ref{general-comparison}
completes the proof.
\end{proof} 
 
\begin{proposition}\label{ZB}
There exists $ \kappa>0$ such that, for all
  $x'\in X'$ and $\rho>0$,  there is a point  $z' \in X'$ such that 
\begin{equation}\label{ZssB}
Z_\b(x',  (\kappa \rho)^{1/\a})\subset B(z',  \rho). 
\end{equation} 
In particular, $\kappa m_{G'}\le  m_{\a,\b}$. 
\end{proposition}

\begin{proof} We choose $\eta\in (0,1)$ such that $(3\eta)^{\a/\b}< \Phi_1(0)$ and define
$\kappa:=\eta^{\a/\b}$.
Let  $x'=(x,r)\in X'$, $\rho>0$, $z':=(x,r+2\eta\rho^{\b/\a})$ and $y'=(y,s)\in Z_\b(x',(\kappa\rho)^{1/\a})$.
Then $|r-s|<(\kappa\rho^{1/\a})^\b=\eta \rho^{\b/\a}$ and 
$t:=(r+2\eta\rho^{\b/\a})-s\in (\eta\rho^{\b/\a},3\eta\rho^{\b/\a})$. Hence 
\begin{equation*}
\sigma:=     \frac{d(x,y)}{t^{1/\b}}
<\frac{(\kappa\rho)^{1/\a}} {(\eta\rho^{\b/\a})^{1/\b}}=1
\end{equation*} 
and
\begin{equation*} 
    {G'}(x',y')    =p_t(x,y) \ge t^{-\a/\b} \Phi_1( \sigma) \ge (3\eta)^{-\a/\b} \rho\inv \Phi_1(1)>\rho\inv.
\end{equation*} 
So $y'\in B(x',\rho)$ proving  (\ref{ZssB}).  Again   an application of Proposition \ref{general-comparison} completes the proof.
\end{proof}

In particular,     
 Theorem \ref{all-para} and Proposition~\ref{ZB} yield the following.

\begin{corollary}\label{final-para}
Every set $A$ in $X'$ with $m_{\a,\b}(A)<\infty$, is contained in 
a~Borel set  which is semipolar, that is, which the process $\mathfrak X'$ 
  hits  at most countably many times.
\end{corollary}

\subsection{Semipolar sets in $X$} 

To complete the paper let us show that we may also generalize our results for classical potential theory 
and Riesz potentials to our setting of heat semigroups on metric measure spaces. To that end
we assume in this section that $\b<\a$, define
\begin{equation*}
            G(x,y):=\int_0^\infty p_t(x,y)\,dt, \qquad x,y\in X, 
\end{equation*} 
and introduce constants  $c,C\in (0,\infty)$ by
\begin{equation*}
 c:=\frac\b{\a-\b} \, \Phi_1(1)  \und C:= \b
\int_0^\infty \sigma^{\a-\b}\,\Phi_2(\sigma)\,\frac{d\sigma}\sigma .
\end{equation*}

\begin{proposition}\label{G-continuous}
$G$ has the properties {\rm (i) -- (iii)} of Section \ref{G-section}
and
\begin{equation}\label{G-estimate}
                c d(x,y)^{-(\a-\b)} \le G(x,y)\le C d(x,y)^{-(\a-\b)}, \qquad x,y\in X.
\end{equation} 
\end{proposition}

\begin{proof} 
Since $\Phi_1(0)\ge \Phi_1(1)>0$, we know that $G=\infty$ 
 on the diagonal of $X\times X$. 
Considering $x,y\in X$, $x\ne y$, and taking $\sigma:=d(x,y) t^{-1/\b}$
we have 
\begin{equation*}
        \frac{d\sigma}{dt}= -\frac 1\b \, t\inv\sigma 
         =- \frac 1\b \,d(x,y)^{-\b}\sigma^{\b+1},  
\end{equation*} 
and hence, using  (\ref{qt}), 
\begin{equation*}
                              \int_0^\infty q_t(x,y)\,dt
=  \b d(x,y)^{-(\a-\b)}
  \int_0^\infty\sigma^{\a-\b}\Phi_2(\sigma)\,\frac{d\sigma}\sigma 
=C d(x,y)^{-(\a-\b)}.
\end{equation*} 
By (\ref{bounds}), we thus conclude that $G(x,y)\le C d(x,y)^{-(\a-\b)}$ and
\begin{equation*}  G(x,y)\ge  \b  d(x,y)^{-(\a-\b)}  \int_0^1 
\sigma^{\a-\b} \Phi_1(\sigma)\, \frac{d\sigma}\sigma \ge c d(x,y)^{-(\a-\b)}.
\end{equation*} 
In particular, (i) and~(iii) hold. 
Moreover, the continuity of $G$ outside the diagonal of $X\times X$  follows
immediately, by Lemma \ref{Leb-prime}.
\end{proof} 

Defining $m_{\a-\b}:=m_{\a-\b, F}$ with $F(x,\rho):=D(x,\rho)$, 
 Theorem~\ref{main-result} and Corollary~\ref{m-polar} already yield  the following.

\begin{theorem}\label{a-b-Hausdorff}
Let $A$ be a subset $X$.
\begin{itemize}
\item[\rm 1.] If  $m_{\a-\b}(A)<\infty$, then $A$ is contained in a $G$-semipolar Borel set.
\item[\rm 2.] If $m_{\a-\b}(A)=0$, then $A$ is contained in a $G$-polar Borel set.
\end{itemize} 
\end{theorem} 

Further, an easy consequence of  (\ref{volume}) and  (\ref{G-estimate})  is the following estimate.

\begin{proposition}\label{Gmu-ball}
Let  $C_\mu:= 2^\a(2^\b-1)\inv c_\mu C$. Then, for all $x\in X$ \hbox{and~$R>0$},
\begin{equation*} 
\int_{D(x,R)} G(x,y)\,d\mu(y)\le C_\mu R^\b.
\end{equation*} 
\end{proposition} 

\begin{proof} 
 Let $x\in X$. For every $r>0$, 
\begin{equation*}
  \int_{D(x,r)\setminus D(x,r/2)} d(x,y)^{-(\a-\b)}\,d\mu(y)\le   \left(\frac r2\right)^{-(\a-\b)} \mu(D(x,r))\le
 2^{\a-\b} c_\mu r^\b,
\end{equation*} 
 and hence, for every $R>0$,
\begin{equation*} 
\int_{D(x,R)}  G(x,y)\,d\mu(y)\le 2^{\a-\b} c_\mu C \sum_{j=0}^\infty (2^{-j}R)^\b=C_\mu R^\b.
\end{equation*}
\end{proof} 

For every $t>0$, let
\begin{equation*}
            P_t f(x):=\int p_t(x,y)f(y)\,d\mu(y), \qquad f\in \B^+(X), \, x\in X.
\end{equation*} 
By \cite[Lemma 8.2]{dense-para}, $\mathbbm P:=(P_t)_{t>0}$ is a strong Feller sub-Markov semigroup on $X$
mapping $\C_0(X)$ into $C_0(X)$ and satisfying $\lim_{t\to 0} P_tf=f $ for every $f\in\C_0(X)$. 

Let $V$ denote the potential kernel of $\mathbbm P$, that is, for every $f\in \B^+(X)$, 
\begin{equation*} 
Vf:=\int_0^\infty P_tf\,dt= G(f\mu).
\end{equation*} 

\begin{proposition}\label{V-feller}
If  $f$ is a bounded function in $\B^+(X)$ with  compact support, 
then $Vf\in \C_0(X)$.
\end{proposition} 

\begin{proof} Given such a function $f$,  there exist $x_0\in X$ and $R>0$
such that $f$ is bounded by a multiple $a f_0$ of  $f_0:=1_{D(x_0,R)}$.
Since both $Vf$ and $V(af_0-f)$ are lower semicontinuous, it hence suffices to show that
$Vf_0\in \C_0(X)$.

For all $N>0$ and $x\in X\setminus D(x_0, R+N))$, 
\begin{equation*} 
                                  G(f_0\mu)(x)\le C_\mu \int_{D(x_0,R)} d(x,y)^{-(\a-\b)}\,d\mu(y)\le C_\mu N^{-(\a-\b)} \mu(D(x_0,R)).
\end{equation*} 
So $Vf_0=G(f_0\mu)$ vanishes at infinity.

Next let $x\in X$ and $\ve>0$. By Proposition \ref{Gmu-ball}, there exists $r>0$ such that 
$V1_{D(z,2r)}(z)< \ve/3$ for every $z\in X$. Let
\begin{equation*} 
g:=1_{D(x_0,R)\setminus D(x,r)}.
\end{equation*} 
 By the continuity of $G$ outside the diagonal, the function $Vg$ is continuous and real on $D(x,r)$. 
So there exists $0<\delta<r $ such that $|Vg(z)-Vg(x)|<\ve/3$ for every $z\in D(x,\delta)$.
Finally, let us fix  $z\in D(x,\delta)$. Then $D(x,r)\subset D(z,2r)$, and hence
\begin{equation*} 
                         | Vf_0(z)-Vf_0(x)|<\frac \ve 3+|Vg(z)-Vg(x)| + \frac \ve 3<\ve.
\end{equation*} 
Thus $Vf$ is continuous at $x$. 
\end{proof} 

Let us choose a   sequence $(g_k)$ in $C_0(X)$ such that $\bigcup_{k=1}^\infty \{g_k>0\}=X$.
There exist $a_k>0$ such that $g_k+ Vg_k\le a_k$, $k\in\nat$. Then $g:=\sum_{k=1}^\infty 2^{-k}a_k\inv g_k
\in \C_0(X)$, $g>0$ and $u:=Vg\in E_{\mathbbm P}\cap\C_0(X)$, $u>0$. So the kernel $V$ is proper and 
we obtain the following; see, for example, \cite[Corollary 2.3.8,2]{H-course}. 

\begin{corollary}\label{balayage-space} 
$(X, E_{\mathbbm P})$ is a balayage space.
\end{corollary} 

Moreover, we know that there exists a Hunt process $\mathfrak X$ on $X$ with transition semigroup $\mathbbm P$;
see \cite[IV.8.1]{BH}. 

\begin{proposition}\label{G-harmonic}
For every $y\in X$, the function $G(\cdot,y)$ is a potential on $X$ with superharmonic support~$\{y\}$.
\end{proposition} 

\begin{proof} Let $y\in X$.  It is straightforward to show that $G(\cdot,y)\in E_{\mathbbm P}$.
Indeed, it suffices to note that, by  Fubini's theorem and (CK), 
\begin{equation*}
  P_sG(\cdot,y)(x)=\int_0^\infty P_s p_t(\cdot,y)(x)\,dt=\int_0^\infty p_{s+t}(x,y)\,dt
=\int_s^\infty p_t(x,y)\,dt 
\end{equation*} 
for all $x\in X$ and $s>0$.
Moreover, clearly
\begin{equation*}
    \inf\{ v\in E_{\mathbbm P} \colon v\ge G(\cdot, y) \mbox{ outside a compact in $X$}\} =0
\end{equation*} 
that is, $G(\cdot,y)$ is a potential. 
To show that $G(\cdot,y)$ is harmonic on $X\setminus \{y\}$ let us fix $x\in X$, $x\ne y$,
and a relatively compact open neighborhood $U$ of $x$ such that $y\notin \ov U$.
Let $\mu_x^U$ be the harmonic measure for $U$ and $x$, that is, for every $v\in E_{\mathbbm P}$,
\begin{equation*}
                \int v\,d\mu_x^U= R_v^{X\setminus U}(x):=\inf\{w(x)\colon w\in E_{\mathbbm P},\, w\ge v\mbox{ on } X\setminus U\},
\end{equation*} 
and $\mu_x^U$ is supported by $X\setminus U$.
Clearly, $\int G(\cdot,y)\,d\mu_x^U\le G(x,y)$.

Suppose that  $\int G(\cdot,y)\,d\mu_x^U< G(x,y)$.
Then, by continuity of $G$ outside the diagonal, there exists $r>0$ such that $D(y,r)\cap U=\emptyset$
and  $\int G(\cdot,z)\,d\mu_x^U< G(x,z)$ for every $z\in D(y,r)$. Since $\mu(D(y,r))>0$,  
integration  with respect to~$\mu$ on~$D(y,r)$ and  Fubini's theorem yield that  the function $v:=V1_{D(y,r)}\in E_{\mathbbm P}$ satisfies 
\begin{equation}\label{contra}
                          \int  v\,d\mu_x^U < v (x).
\end{equation} 
However, by \cite[II.7.1]{BH}, 
\begin{equation*}
                            \int v\,d\mu_x^U=  R_v^{X\setminus U}(x)\ge   R_v^{ D(y,r)} (x)=v(x)
\end{equation*} 
contradicting (\ref{contra}).  Thus 
 $\int G(\cdot,y)\,d\mu_x^U= G(x,y)$ completing the proof.
\end{proof} 

We claim that no point is finely isolated. So let $y\in X$. Clearly,  $R_1^{\{y\}}\le a G(\cdot,y)$ for every $a>0$, 
 and hence $R_1^{\{y\}}=1_{\{y\}}$. Knowing already that $y$ is not isolated we obtain that
$\hat R_1^{\{y\}}(y)=\liminf_{z\to y} R_1^{\{y\}}(z)=0$.  
This implies that $\{y\}$ is not finely isolated, since $\hat R_1^{\{y\}}(y)$ is also the fine lower limit of 
$ R_1^{\{y\}}$ at $y$; see \cite[III.5.9]{BH}.

By  \cite[Theorem~1.1]{HN-representation}, we hence obtain that $G$ has   property~(iv)
from Section~\ref{habasp}. Proposition \ref{G-harmonic} and (\ref{G-estimate}) show that also    (i$'$), (ii) and (iii) hold.
By Theorem~\ref{a-b-Hausdorff} and  Proposition \ref{G-semipolar-semipolar},
we therefore conclude the following.

\begin{corollary}\label{final}
Every set $A$ in $X$ with $m_{\a-\b}(A)<\infty$, is contained in  
a~Borel set  which is semipolar, that is, which the process $\mathfrak X$ 
  hits  at most countably many times.
\end{corollary}

{\small \noindent 
Wolfhard Hansen,
Fakult\"at f\"ur Mathematik,
Universit\"at Bielefeld,
33501 Bielefeld, Germany, e-mail:
 hansen$@$math.uni-bielefeld.de}\\
{\small \noindent Ivan Netuka,
Charles University,
Faculty of Mathematics and Physics,
Mathematical Institute,
 Sokolovsk\'a 83,
 186 75 Praha 8, Czech Republic, email:
netuka@karlin.mff.cuni.cz}

\end{document}